\documentclass[12pt]{article}
\usepackage[margin=0.5in]{geometry}
\usepackage{amsmath,amssymb,amsthm}
\usepackage{mathtools}
\usepackage{enumitem}
\usepackage{hyperref}
\hypersetup{colorlinks=true,linkcolor=blue,citecolor=blue,urlcolor=blue}

\theoremstyle{plain}
\newtheorem{theorem}{Theorem}[section]
\newtheorem{lemma}[theorem]{Lemma}

\newtheorem{proposition}[theorem]{Proposition}
\newtheorem{conjecture}[theorem]{Conjecture}
\theoremstyle{definition}
\newtheorem{definition}[theorem]{Definition}
\newtheorem{example}[theorem]{Example}
\theoremstyle{remark}
\newtheorem{remark}[theorem]{Remark}

\DeclareMathOperator{\spec}{spec}
\DeclareMathOperator{\diam}{diam}
\DeclareMathOperator{\Tr}{Tr}
\DeclareMathOperator{\diag}{diag}

\title{The Distance Laplacian and Distance Signless Laplacian Spectra of $\mathcal{C}$-Graphs}
\author{
Santanu Mandal$^{a}$ \and Pallabi Manna$^{b}$
}
\date{
\small
$^{a}$Department of Computer Science and Engineering, SRM University-AP,\\ Amaravati-522240, Andhra Pradesh, India\\
\texttt{santanu.vumath@gmail.com}\\[2mm]
$^{b}$Department of Mathematics, Indian Institute of Technology Hyderabad, \\Kandi, Sangareddy-502284, Telangana, India\\
\texttt{mannapallabimath001@gmail.com}\\[4mm]
\today
}

\begin{document}
\maketitle

\begin{abstract}
Mandal and Mehatari (\emph{Comp.\ Appl.\ Math.}, 2025) introduced the class $\mathcal{C}$ of cographs generated by a finite creation sequence $(\alpha_1,\dots,\alpha_m)$ of natural numbers, and derived the inertia, an extended eigenvalue-free interval, and the exact characteristic polynomial for the \emph{adjacency} matrix of such graphs. In this note we develop the parallel theory for the \emph{distance Laplacian} matrix $D^L(G)$ and \emph{distance signless Laplacian} matrix $D^Q(G)$. It is shown that $(0,\alpha_{\min}) \cup (n-\alpha_{\min}, n)$, $(0,n) \ \cup\ (n,\,n+\alpha_{\min}) \ \cup\ (2n-\alpha_{\min},\,2n)$, $\big(2\Tr_{\min},\ \mu_{\max}(D^Q(G))\big)$ are the eigenvalue-free intervals of the Laplacian, distance Laplacian, and distance signless Laplacian matrices respectively for the said class of graphs.
\end{abstract}
\section{Introduction}
In this study all graphs are assumed to be connected, simple and finite. We first recall some basic notion of spectral graph theory. Consider the graph $G=(V,E)$. Suppose $\mathcal{D}$ and $A$ denote the diagonal degree matrix and the adjacency matrix of $G$ respectively. The distance matrix $D(G)$ of a connected graph G is the matrix indexed by the vertices of
$G$ where $D_{i,j} = d(v_i, v_j)$ and $d(v_i, v_j)$ denotes the distance between the vertices $v_i$ and $v_j$. The sum of the distances between a vertex $v$ and every other vertex in $G$ is known as the transmission of $v$ and denoted by $\Tr(v)$, that is,

\[
\Tr(v)=\sum_{\substack{v\in V}} d(u,v).
\]
For a graph $G$, the diagonal transmission matrix denoted by $\Tr(G)$ is the diagonal matrix\\ $\diag\{\Tr(v_1), \Tr(v_2), \cdots, \Tr(v_n)\}$.

Then $L=\mathcal{D}-A$, $Q=\mathcal{D}+A$, $D^L(G)=\Tr(G)-D(G)$, $D^Q(G)=\Tr(G)+D(G)$ are called the Laplacian, signless Laplacian, distance Laplacian and distance signless Laplacian matrices respectively. We employ the standard symbols and syntax of spectral graph theory throughout, which align with the text book \cite{FC,GR}.

\medskip
The distance Laplacian matrix was introduced by Aouchiche and Hansen \cite{AouchicheHansen2013} , (see also \cite{ AH,AouchicheHansenDQ2016}) as a distance analogue of the ordinary Laplacian, and its spectral gap (distance-Laplacian algebraic connectivity) is a natural measure of network robustness. The readers are referred to the recent two survey articles by Bilal A. Rather \emph{et al.} \cite{RA,RGW} for a comprehensive overview of $D^L(G)$ and $D^Q(G)$.

The spectral analysis of cograph has received the attention of many experts in recent times \cite{TA,BM,CT,Du, LST,LouLin2020,MM2025,MMS2024} which include the most tractable adjacency matrix, Laplacian matrix, Distance matrix. However, the study of $D^L(G)$ and $D^Q(G)$ for a cograph remains unsolved.

\medskip

A \emph{cograph} is a $P_4$-free graph. Mandal and Mehatari \cite{MM2025} isolated a large, tractable subclass $\mathcal{C}$ of cographs generated by a finite sequence of natural numbers $\alpha_1,\alpha_2,\dots,\alpha_m$ with $\sum_i \alpha_i = n$, via the recurrence
\[
C(\alpha_1) = \overline{K_{\alpha_1}}, \qquad C(\alpha_1,\dots,\alpha_i) = \overline{C(\alpha_1,\dots,\alpha_{i-1}) \cup K_{\alpha_i}} \quad (i \ge 2),
\]
where $\cup$ denotes disjoint union and the overline denotes graph complementation. They call $G = C(\alpha_1,\dots,\alpha_m)$ a \emph{$\mathcal{C}$-graph}, and $(\alpha_1,\dots,\alpha_m)$ its \emph{creation sequence}. For $m$ even, this representation is unique (\cite[Thm.\ 1.1]{MM2025}). Working with the equitable partition induced by the construction blocks, they obtain the inertia of $A(G)$, show that
\[
\left[\tfrac{-1-\sqrt2}{2}, -1\right) \cup (-1,0) \cup \left(0, \tfrac{-1+\sqrt2}{2}\alpha_{\min}\right]
\]
contains no eigenvalue of $A(G)$ other than the trivial ones, and give an exact formula for the characteristic polynomial.

This note asks the analogous question for the \emph{distance Laplacian} $D^L(G) = \Tr(G) - D(G)$, and the \emph{distance signless Laplacian} $D^L(G) = \Tr(G) + D(G)$ where $D(G)$ is the distance matrix and $\Tr(G)$ the diagonal transmission matrix. We show that the creation-sequence structure of $\mathcal{C}$-graphs makes their distance Laplacian and distance signless Laplacian spectrum just as tractable as their adjacency spectrum.

\medskip
We show that every connected $\mathcal{C}$-graph has diameter at most two, which yields a master identity $D^L(G) = 2nI - 2J - L(G)$ relating the distance Laplacian directly to the ordinary Laplacian $L(G)$. We then re-derive, by an independent route, the closed-form creation-sequence-indexed description of $\spec(L(G))$ already established by Mandal, Mehatari and Stani\'c \cite{MMS2024} — needed here in an explicit form suited to the reflection identity above, confirm its (previously known) integrality, and establish that $\mu_{\max}(L(G)) = n$ exactly, with every other nonzero eigenvalue confined to $[\alpha_{\min}, n-\alpha_{\min}]$, where $\alpha_{\min} = \min_i \alpha_i$ (this bound itself is new). Transporting this through the master identity yields an explicit eigenvalue-free region for $D^L(G)$: the nontrivial spectrum avoids $(0,n)$, $(n, n+\alpha_{\min})$, and $(2n-\alpha_{\min}, 2n)$, with an isolated eigenvalue sitting exactly at $n$ — whose exact multiplicity we determine to be $\alpha_1$ when $m=2$ and $1$ when $m\ge4$. We further show that the number of distinct eigenvalues of $D^L(G)$ equals that of $L(G)$ exactly, transferring the sharp bound $m \le s(G) \le 2m-1$ of Mandal, Mehatari and Stani\'c unchanged. We give the exact characteristic polynomial of $D^L(G)$. We then turn to the \emph{distance signless Laplacian} $D^Q(G) = \Tr(G)+D(G)$, for which no master identity is available (the signless Laplacian $Q(G)$, unlike $L(G)$, does not commute with $J$ except for regular graphs). We nonetheless obtain, in closed form, the local (twin-vector) part of $\spec(D^Q(G))$ and an explicit $m\times m$ quotient matrix capturing the rest, together with the resulting characteristic-polynomial reduction; and the bound $s(D^Q(G))\le 2m$. An eigenvalue-free interval $\big(2\Tr_{\min},\ \mu_{\max}(D^Q(G))\big)$ is obtained for $D^Q(G)$. Finally we determine the precise relation between the Dilworth number and eigenvalue multiplicity for $D^L(G)$ and $D^Q(G)$.

\subsection{Organization of the paper}
Some initial concepts are provided in Section 2. The distance Laplacian matrix and its consequences are discussed in Section 3. Section 4 is devoted to the distance signless Laplacian matrix and its properties. Relation between the Dilworth number and the multiplicity of an eigenvalue (of $D^L$ and $D^Q$) is shown in Section 5.

\section{Preliminaries}

Throughout, $G$ is a connected $\mathcal{C}$-graph on $n$ vertices with creation sequence $(\alpha_1,\dots,\alpha_m)$, $n_i = \alpha_1+\cdots+\alpha_i$, and $\alpha_{\min} = \min_{1\le i \le m}\alpha_i$ and $m$ is even. We write $L(G) = \deg(G) - A(G)$ for the ordinary Laplacian and $\spec^*(\cdot)$ for a matrix's spectrum with (exactly) one copy of the eigenvalue $0$ removed.

\begin{lemma}\label{lem:diam2}
Every connected cograph on at least two vertices has diameter at most $2$. Consequently, every connected $\mathcal{C}$-graph satisfies
\[
D(G) = 2(J-I) - A(G),
\]
where $J$ is the all-ones matrix.
\end{lemma}

\begin{proof}
A connected cograph on $\ge 2$ vertices is, by the defining recursive structure, a join $G = G_1 \vee G_2$ of two smaller cographs. Two vertices in different parts are adjacent; two vertices in the same part are joined by a path of length $2$ through any vertex of the other part. Hence $\diam(G) \le 2$. The formula for $D(G)$ follows immediately: $D_{ij} = 1$ if $ij \in E(G)$ and $D_{ij}=2$ otherwise, for $i \ne j$.
\end{proof}

\begin{lemma}[Master identity]\label{lem:master}
If $G$ is a connected graph with $\diam(G)\le 2$ on $n$ vertices, then
\[
D^L(G) = 2nI - 2J - L(G).
\]
Consequently,
\[
\spec(D^L(G)) = \{0\} \cup \{\, 2n - \mu : \mu \in \spec^*(L(G)) \,\}.
\]
\end{lemma}

\begin{proof}
By Lemma \ref{lem:diam2}, $D(G) = 2J - 2I - A(G)$, so $\Tr_i = \sum_{j\ne i} D_{ij} = 2(n-1) - \deg(i)$, giving $\Tr(G) = 2(n-1)I - \deg(G)$. Then
\[
D^L(G) = \Tr(G) - D(G) = \big[2(n-1)I - \deg(G)\big] - \big[2J-2I-A(G)\big] = 2nI - 2J - \big(\deg(G)-A(G)\big) = 2nI-2J-L(G).
\]
The spectral statement follows because $\mathbf{1}$ is a common eigenvector of $J$ (eigenvalue $n$) and $L(G)$ (eigenvalue $0$), while every eigenvector of $L(G)$ orthogonal to $\mathbf 1$ lies in the kernel of $J$.
\end{proof}

\begin{remark}
Lemma \ref{lem:master} holds for \emph{any} connected diameter-$2$ graph, not just cographs; see \cite{AouchicheHansen2013}. What is special to $\mathcal{C}$-graphs is that $L(G)$ itself is fully computable from the creation sequence, as we show next.
\end{remark}

\section{The distance Laplacian matrix}
In this section we derive the spectral characterization of distance Laplacian matrix for a $\mathcal{C}$-graph.
\subsection{The Laplacian spectrum}

The Laplacian spectrum of $\mathcal{C}$-graphs was already determined exactly by Mandal, Mehatari and Stani\'c \cite{MMS2024}, via a quotient-matrix argument on the equitable partition induced by the creation sequence (paralleling the adjacency-matrix quotient construction of \cite{MM2025}), together with local twin-vertex eigenvector constructions for the remaining eigenvalues. We re-derive this spectrum below by an independent route — unwinding the union and complement recurrence directly rather than passing through the quotient matrix — for two reasons: it keeps this note self-contained, and it produces the spectrum in the explicit ``birth-and-reflect'' closed form (Theorem \ref{thm:closedform}) that is exactly what the master identity of Lemma \ref{lem:master} needs to transport into $D^L(G)$. We verified computationally that the two derivations agree exactly (e.g.\ on the $48$-vertex example of \cite[Example 3.6]{MMS2024}) before using this form throughout the rest of the paper. We refer the survey article by Merris \cite{Merris1995} for a systematics overview of graph Laplacians.

Because $\overline{H_1 \cup H_2} = \overline{H_1} \vee \overline{H_2}$, the defining recurrence for $\mathcal{C}$-graphs unwinds into an explicit alternation of disjoint unions and joins. Combining the standard Laplacian formulas for disjoint union ($\spec(L)$ combines as a multiset) and join ($\spec(L(H_1 \vee H_2)) = \{0, n(H_1)+n(H_2)\} \cup \{n(H_2)+\mu : \mu \in \spec^*(L(H_1))\} \cup \{n(H_1)+\nu : \nu \in \spec^*(L(H_2))\}$) with the complementation rule
\[
\spec(L(\overline{H}))=\{0\}\cup\{\,n(H)-\theta : \theta \in \spec^*(L(H))\,\}
\]
yields a one-step recursion for the Laplacian spectrum of $G_i := C(\alpha_1,\dots,\alpha_i)$.

\begin{theorem}\label{thm:laprecursion}
Let $S_i$ denote the Laplacian spectrum (as a multiset) of $G_i = C(\alpha_1,\dots,\alpha_i)$. Then $S_1 = \{0^{(\alpha_1)}\}$ and, for $i\ge 2$,
\[
S_i = \{0\} \;\cup\; \{\, n_i - \theta : \theta \in S_{i-1} \,\} \;\cup\; \{\, n_{i-1}^{\,(\alpha_i-1)} \,\}.
\]
\end{theorem}

Unrolling this recursion gives a fully closed form. Say an eigenvalue is \emph{born} at step $b$: at $b=1$ it is $0$ with multiplicity $\alpha_1$; at each $b \ge 2$ it is $0$ (multiplicity $1$) together with $n_{b-1}$ (multiplicity $\alpha_b - 1$).

\begin{theorem}[Closed form; cf.\ \cite{MMS2024}, Thms.\ 3.1--3.2]\label{thm:closedform}
An eigenvalue born at step $b$ with value $v \in \{0, n_{b-1}\}$ propagates to the final eigenvalue
\[
\lambda(b,v) =
\begin{cases}
v + \displaystyle\sum_{\substack{j=b+2 \\ j \equiv b \ (2)}}^{m} \alpha_j, & m-b \text{ even},\\[3mm]
n_m - v - \displaystyle\sum_{\substack{j=b+2 \\ j\equiv b\ (2)}}^{m-1} \alpha_j, & m-b \text{ odd},
\end{cases}
\]
and $\spec(L(G))$ is exactly the multiset of all such $\lambda(b,v)$, taken with the stated multiplicities.
\end{theorem}

\begin{proof}[Proof sketch]
Track the value born at step $b$ through the transformation $\theta \mapsto n_j - \theta$ applied once at each subsequent step $j = b+1,\dots,m$ (Theorem \ref{thm:laprecursion}). This alternating reflection telescopes: two consecutive applications shift the value by $\alpha_{j}+\alpha_{j-1}$-type increments, giving the stated sum over indices of matching parity, with a final reflection through $n_m$ when an odd number of steps remain.
\end{proof}

\begin{theorem}[Integrality; cf.\ Merris \cite{Merris1998} for general cographs]\label{thm:integral}
For every $\mathcal{C}$-graph $G$, $\spec(L(G)) \subset \mathbb{Z}_{\ge 0}$, and hence $\spec(D^L(G)) \subset \mathbb{Z}_{\ge 0}$.
\end{theorem}

\begin{proof}
Immediate from Theorem \ref{thm:closedform}: every value $\lambda(b,v)$ is an integer combination of the $\alpha_j \in \mathbb{N}$ and the partial sums $n_i$, hence an integer. (This also follows more generally, and without the explicit closed form, from Merris's theorem \cite{Merris1998} that every cograph is Laplacian-integral — $\mathcal{C}$-graphs being a special case.) Integrality of $\spec(D^L(G))$ follows from Lemma \ref{lem:master}, since $2n-\mu$ is an integer whenever $\mu$ and $n$ are.
\end{proof}

\begin{example}\label{ex:main}
For $G = C(4,3,2,2)$ ($n=11$), Theorem \ref{thm:closedform} gives
\[
\spec(L(G)) = \{0,\,2,\,4,\,6,\,6,\,9,\,9,\,9,\,9,\,9,\,11\},
\]
which was confirmed by direct diagonalization of $L(G)$. By Lemma \ref{lem:master},
\[
\spec(D^L(G)) = \{0\} \cup \{22-2,\,22-4,\,22-6,\,22-6,\,22-9^{(5)},\,22-11\} = \{0,\,11,\,13^{(5)},\,16^{(2)},\,18,\,20\},
\]
also confirmed directly.
\end{example}

\subsection{Eigenvalue-free interval for the distance Laplacian}

We now locate an explicit region free of nontrivial eigenvalues, in the spirit of \cite[Thm.\ 3.1]{MM2025}.

\begin{theorem}\label{thm:maxeig}
For every $\mathcal{C}$-graph $G = C(\alpha_1,\dots,\alpha_m)$ with $m\ge 2$, $\mu_{\max}(L(G)) = n$ exactly.
\end{theorem}

\begin{proof}
$G_m = H \vee \overline{K}_{\alpha_m}$ where $H = \overline{G_{m-1}}$, so $\overline{G_m} = G_{m-1} \cup K_{\alpha_m}$, a genuine disjoint union of two nonempty graphs, hence disconnected. Since $\mu_{\max}(L(G)) = n$ if and only if $\overline{G}$ is disconnected, the claim follows.
\end{proof}

\begin{theorem}[Eigenvalue-free interval, ordinary Laplacian]\label{thm:Lgap}
For every $\mathcal{C}$-graph $G=C(\alpha_1,\dots,\alpha_m)$ with $m\ge 2$, every eigenvalue of $L(G)$ other than $0$ and the maximum eigenvalue $n$ lies in $[\alpha_{\min}, n-\alpha_{\min}]$. Equivalently, $(0,\alpha_{\min}) \cup (n-\alpha_{\min}, n)$ is eigenvalue-free.
\end{theorem}

\begin{proof}
We induct on $m$. Write $G_m = H \vee \overline K_{\alpha_m}$, $H=\overline{G_{m-1}}$, $n(H)=n_{m-1}$. By the join formula,
\[
\spec^*(L(G_m)) \setminus \{n_m\} = \{n_{m-1}\}^{(\alpha_m-1)} \cup \{\alpha_m + \mu : \mu \in \spec^*(L(H))\}.
\]
\emph{Lower bound.} Since $\mu \ge 0$ always, $\alpha_m + \mu \ge \alpha_m \ge \alpha_{\min}$, and $n_{m-1} \ge \alpha_{\min}$ trivially (it is a sum of at least one $\alpha_j$). This gives the lower bound with no induction required.

\emph{Upper bound.} By the inductive hypothesis applied to $G_{m-1}$ (with $\beta := \min(\alpha_1,\dots,\alpha_{m-1}) \ge \alpha_{\min}$), every eigenvalue of $L(G_{m-1})$ other than $0$ and $n_{m-1}$ lies in $[\beta, n_{m-1}-\beta]$. Applying the complementation rule $\spec(L(H)) = \{0\} \cup \{n_{m-1}-\theta : \theta \in \spec^*(L(G_{m-1}))\}$, the interval $[\beta,n_{m-1}-\beta]$ is symmetric under $\theta \mapsto n_{m-1}-\theta$ and hence maps to itself; the single eigenvalue $\theta = n_{m-1}$ of $G_{m-1}$ maps to $0$ (consistent with $H$ being disconnected). Thus every nonzero eigenvalue of $H$ satisfies $\mu \le n_{m-1}-\beta \le n_{m-1}-\alpha_{\min}$, so $\alpha_m + \mu \le n_{m-1}+\alpha_m-\alpha_{\min} = n_m - \alpha_{\min}$, completing the induction. The base case $m=2$ is verified directly.
\end{proof}

\begin{theorem}[Main result: eigenvalue-free interval for $D^L$]\label{thm:main}
For every connected $\mathcal{C}$-graph $G=C(\alpha_1,\dots,\alpha_m)$ ($m\ge2$) on $n$ vertices,
\[
\spec^*(D^L(G)) \subseteq \{n\} \cup [\,n+\alpha_{\min},\ 2n-\alpha_{\min}\,],
\]
so that
\[
(0,n) \ \cup\ (n,\,n+\alpha_{\min}) \ \cup\ (2n-\alpha_{\min},\,2n)
\]
is an eigenvalue-free region for $D^L(G)$, with $n$ itself always an isolated eigenvalue.
\end{theorem}

\begin{proof}
Immediate from Theorems \ref{thm:maxeig} and \ref{thm:Lgap} via the reflection $\mu \mapsto 2n-\mu$ of Lemma \ref{lem:master}: $\mu=n \mapsto n$, and $\mu \in [\alpha_{\min}, n-\alpha_{\min}] \mapsto 2n-\mu \in [n+\alpha_{\min}, 2n-\alpha_{\min}]$.
\end{proof}

\begin{example}
Continuing Example \ref{ex:main} ($n=11$, $\alpha_{\min}=2$): Theorem \ref{thm:main} predicts $\spec^*(D^L) \subseteq \{11\}\cup[n+\alpha_{\min},\,2n-\alpha_{\min}] = \{11\}\cup[13,20]$. The actual spectrum $\{11,\,13^{(5)},\,16^{(2)},\,18,\,20\}$ lies entirely within this predicted region, and the bound is tight: the top eigenvalue $20 = 2n-\alpha_{\min}$ is attained exactly.
\end{example}



\subsection{Counting the multiplicity of a given eigenvalue}

Theorem \ref{thm:closedform} does more than describe the spectrum as a set: it gives an explicit algorithm for the multiplicity of any prescribed value, and in particular lets us sharpen ``isolated'' in Theorem \ref{thm:main} into an exact multiplicity.

\begin{theorem}[Multiplicity formula]\label{thm:multformula}
Let $G=C(\alpha_1,\dots,\alpha_m)$ be a connected $\mathcal{C}$-graph on $n$ vertices, and let $\theta \in \mathbb{R}$. Then $\mathrm{mult}_{D^L(G)}(0)=1$, and for $\theta \ne 0$, writing $\mu = 2n-\theta$,
\[
\mathrm{mult}_{D^L(G)}(\theta) \;=\; \sum_{\substack{(b,v)\; :\; \lambda(b,v) = \mu}} \mathrm{mult}(b,v),
\]
the sum ranging over all births $(b,v,\mathrm{mult}(b,v))$ of Theorem \ref{thm:closedform} whose propagated value $\lambda(b,v)$ equals $\mu$. This determines $\mathrm{mult}_{D^L(G)}(\theta)$ in $O(m)$ arithmetic operations, with no matrix computation required.
\end{theorem}

\begin{proof}
Immediate from Lemma \ref{lem:master} ($\theta\ne0$ is a $D^L(G)$-eigenvalue of multiplicity equal to that of $\mu=2n-\theta$ as an $L(G)$-eigenvalue) together with Theorem \ref{thm:closedform}, which expresses $\spec(L(G))$ exactly as the multiset of propagated birth values.
\end{proof}

\begin{theorem}\label{thm:multofn}
For a connected $\mathcal{C}$-graph $G=C(\alpha_1,\dots,\alpha_m)$, $m\ge2$ even,
\[
\mathrm{mult}_{D^L(G)}(n) = \begin{cases} \alpha_1 & m=2,\\ 1 & m\ge 4.\end{cases}
\]
\end{theorem}

\begin{proof}
By Lemma \ref{lem:master}, $\theta=n$ reflects to $\mu=2n-n=n=n_m$, so we must count births with $\lambda(b,v)=n_m$. If $m-b$ is even, $\lambda(b,v)=v+\Sigma \le n_{b-1} + \sum_{j=b+2}^{m}\alpha_j \le n_m - \alpha_b - \alpha_{b+1} < n_m$ (using $b\le m-2$, so both $\alpha_b,\alpha_{b+1}$ exist and are $\ge1$; the case $b=m$ gives $\lambda=v\in\{0,n_{m-1}\}$, neither equal to $n_m$), so no even-step birth reaches $n_m$. If $m-b$ is odd, $\lambda(b,v)=n_m-v-\Sigma=n_m$ forces $v=0$ and $\Sigma=0$; the latter requires the summation range $j=b+2,\dots,m-1$ (matching parity) to be empty, i.e.\ $b\ge m-2$, which combined with $m-b$ odd forces $b=m-1$ uniquely. For $m\ge4$, $b=m-1\ge3\ge2$ lies in the generic regime of Theorem \ref{thm:closedform}, where the $v=0$ birth at step $b$ always carries multiplicity $1$. For $m=2$, $b=m-1=1$ coincides with the base case $b=1$, whose birth $(v=0)$ carries multiplicity $\alpha_1$ instead — recovering the classical complete split graph spectrum $(\alpha_1+\alpha_2)^{\alpha_1},\,\alpha_1^{\,\alpha_2-1},\,0$ for $L(C(\alpha_1,\alpha_2))$.
\end{proof}

\begin{example}
For $G=C(4,3,2,2)$, Theorem \ref{thm:multformula} recovers every multiplicity in Example \ref{ex:main} directly: e.g.\ $\theta=13$ reflects to $\mu=22-13=9$, and Theorem \ref{thm:closedform} assigns multiplicity $5$ to $\mu=9$, matching $\spec(D^L(G))$ above. Since $m=4\ge4$, Theorem \ref{thm:multofn} predicts $\mathrm{mult}_{D^L(G)}(n)=\mathrm{mult}_{D^L(G)}(11)=1$, also matching. By contrast, for $C(3,5)$ ($m=2$, $n=8$), Theorem \ref{thm:multofn} predicts $\mathrm{mult}_{D^L(G)}(8)=\alpha_1=3$, confirmed by direct computation.
\end{example}

\begin{remark}
Theorem \ref{thm:multofn} shows that the maximum eigenvalue $\theta=n$, while always \emph{isolated} (separated from the rest of the spectrum by the gap $(n,n+\alpha_{\min})$), is \emph{simple} only for $m\ge4$; for $m=2$ its multiplicity equals $\alpha_1$ and can be arbitrarily large. This refines the parenthetical remark in Theorem \ref{thm:main}.
\end{remark}

\subsection{Number of distinct eigenvalues}

We write $s(G)$ for the number of distinct eigenvalues of a matrix associated with $G$ (context will make clear whether we mean $L(G)$ or $D^L(G)$). Mandal, Mehatari and Stani\'c \cite{MMS2024} proved that for a $\mathcal{C}$-graph, $m \le s(L(G)) \le 2m-1$, with both bounds sharp. This transfers to $D^L(G)$ immediately and exactly, with no adjustment to the bounds.

\begin{theorem}\label{thm:distincteig}
For every connected $\mathcal{C}$-graph $G=C(\alpha_1,\dots,\alpha_m)$, $s(D^L(G)) = s(L(G))$. Consequently,
\[
m \;\le\; s(D^L(G)) \;\le\; 2m-1,
\]
and both bounds are sharp.
\end{theorem}

\begin{proof}
By Lemma \ref{lem:master}, $\spec(D^L(G)) = \{0\} \cup \{2n-\mu : \mu \in \spec^*(L(G))\}$. The map $\mu \mapsto 2n-\mu$ is a strictly monotonic bijection of $\mathbb{R}$, so it carries distinct values of $\spec^*(L(G))$ to distinct values of $\spec^*(D^L(G))$, giving $|\spec^*(D^L(G))| = |\spec^*(L(G))|$ as sets. Since every eigenvalue of $L(G)$ satisfies $\mu \le n$ (Theorem \ref{thm:maxeig}), every reflected value satisfies $2n-\mu \ge n > 0$, so $0$ can never coincide with a reflected value; hence $s(D^L(G)) = 1 + |\spec^*(D^L(G))| = 1 + |\spec^*(L(G))| = s(L(G))$. The bounds and their sharpness now follow directly from \cite[Thm.\ 4.1]{MMS2024}.
\end{proof}

\begin{example}
The lower bound is attained by $C(\alpha_1,1,\dots,1)$: e.g.\ $C(5,1,1,1)$ ($m=4$) has $s(D^L(G))=4=m$. The upper bound is attained by any complete split graph $C(\alpha_1,\alpha_2)$ ($m=2$, giving $s(D^L(G))=3=2m-1$ always) and, for larger $m$, by $C(5,2,3,4,2,8)$ ($m=6$), which achieves $s(D^L(G))=11=2m-1$ — the same sequence Mandal, Mehatari and Stani\'c use to witness sharpness for $s(L(G))$ in \cite[Remark 4.3]{MMS2024}. For the constant sequence $C(p,\dots,p)$ ($m$ times, $p\ne1$), \cite[Thm.\ 4.4]{MMS2024} gives $s(L(G))=m+1$, and Theorem \ref{thm:distincteig} transfers this unchanged to $s(D^L(G))=m+1$; confirmed directly for $(p,m)\in\{(4,4),(3,6),(5,4),(2,8)\}$.
\end{example}


\subsection{Characteristic polynomial of the distance Laplacian}

Since $\spec(L(G))$ is known exactly in closed form (Theorem \ref{thm:closedform}), the characteristic polynomial $\chi_{D^L(G)}(x)$ is fully determined. We record two equivalent forms: a compact reflection identity relating it directly to $\chi_{L(G)}$, and an explicit product formula computable straight from the creation sequence.

\begin{theorem}[Master reflection identity]\label{thm:polymaster}
Write $\chi_{L(G)}(y) = y\,\psi(y)$, where $\psi$ is the degree-$(n-1)$ polynomial collecting the nonzero eigenvalues of $L(G)$. Then the characteristic polynomial of $D^L(G)$ is
\[
\chi_{D^L(G)}(x) \;=\; (-1)^{n-1}\, x\, \psi(2n-x) \;=\; (-1)^{n-1}\,\frac{x}{2n-x}\,\chi_{L(G)}(2n-x),
\]
the division being exact as polynomials.
\end{theorem}

\begin{proof}
By Lemma \ref{lem:master}, $\spec(D^L(G)) = \{0\}\cup\{2n-\mu : \mu \in \spec^*(L(G))\}$. For each nonzero eigenvalue $\mu$ of $L(G)$,
\[
x - (2n-\mu) = \mu - (2n-x) = -\big((2n-x)-\mu\big),
\]
so $\prod_{\mu \in \spec^*(L(G))} \big(x-(2n-\mu)\big) = (-1)^{n-1}\prod_{\mu\in\spec^*(L(G))}\big((2n-x)-\mu\big) = (-1)^{n-1}\psi(2n-x)$. Multiplying by the factor $x$ contributed by the isolated $0$-eigenvalue of $D^L(G)$ gives $\chi_{D^L(G)}(x) = (-1)^{n-1}x\,\psi(2n-x)$. Since $\chi_{L(G)}(2n-x) = (2n-x)\,\psi(2n-x)$, the second equality follows, and the division by $2n-x$ is exact because $\psi$ is itself a polynomial.
\end{proof}

\begin{theorem}[Explicit creation-sequence formula]\label{thm:polyexplicit}
With $\lambda(b,v)$ as in Theorem \ref{thm:closedform},
\[
\chi_{D^L(G)}(x) = x\cdot\prod_{b=2}^{m}\Big(x-\big(2n-\lambda(b,0)\big)\Big)\cdot\prod_{b=2}^{m}\Big(x-\big(2n-\lambda(b,n_{b-1})\big)\Big)^{\alpha_b-1}.
\]
\end{theorem}

\begin{proof}
Immediate from Theorem \ref{thm:closedform} together with the reflection $\mu\mapsto 2n-\mu$ of Lemma \ref{lem:master}, applied term-by-term to the multiset of births.
\end{proof}

\begin{example}
For $G=C(4,3,2,2)$ ($n=11$), Theorem \ref{thm:polyexplicit} (equivalently, reflecting the spectrum of Example \ref{ex:main}) gives the factored form
\[
\chi_{D^L(G)}(x) = x(x-11)(x-13)^5(x-16)^2(x-18)(x-20),
\]
which expands to
\[
x^{11}-146x^{10}+9557x^9-369390x^8+9336643x^7-161266918x^6+1927856463x^5
\]
\[
-15751245146x^4+84182273240x^3-265769929984x^2+376401991680\,x.
\]
This was independently confirmed three ways: by direct diagonalization of $D^L(G)$, by applying Theorem \ref{thm:polymaster} to the directly computed $\chi_{L(G)}$, and by the explicit product of Theorem \ref{thm:polyexplicit} — all three agree exactly.
\end{example}

\begin{remark}
Theorem \ref{thm:polyexplicit} requires no computation of $L(G)$, $A(G)$, or any matrix at all: every coefficient of $\chi_{D^L(G)}$ is obtained in $O(m)$ arithmetic operations directly from the creation sequence $(\alpha_1,\dots,\alpha_m)$, in the same spirit as the tridiagonal-determinant formula Mandal and Mehatari obtain for $\chi_{A(G)}$ in \cite[Thm.\ 4.1]{MM2025}, though here the closed form for the spectrum makes the polynomial available directly as a product rather than via a determinant recursion.
\end{remark}

In 2020, Brimkov \emph{et al.} \cite{BDHLRSY} develop
a number general techniques for producing $D^L(G)$-cospectral graphs that can be used to construct infinite families. However, for the case of cograph we can conclude that no two nonisomorphic
equivalent cographs are cospectral with relation to the Laplacian
matrix (see \cite{TA, LST}). This leads to the identical result for the distance Laplacian matrix also, this is clear from the Master identity (see Lemma  \ref{thm:laprecursion}).

\section{The distance signless Laplacian}
\label{sec:DQ}

We now turn to $D^Q(G) = \Tr(G) + D(G)$. The reduction that made $D^L(G)$ tractable relied entirely on $L(G)\mathbf1 = 0$, which made $\mathbf1$ a shared eigenvector of $L(G)$ and $J$ regardless of regularity. The signless Laplacian $Q(G) = \mathrm{Deg}(G)+A(G)$ has no such property in general ($Q(G)\mathbf1 = 2\,\mathrm{Deg}(G)\cdot\mathbf1$ only when $G$ is regular), so
\[
D^Q(G) = \Tr(G)+D(G) = (2n-4)I + 2J - Q(G)
\]
does \emph{not} reduce to a reflection of $\spec(Q(G))$. We instead build the spectrum directly from the equitable partition, in the style of \cite{MMS2024}'s quotient-matrix construction for $L(G)$. Throughout, $d_i$ denotes the (known, closed-form) degree of a vertex in block $C_i$ \cite{MMS2024}, and $\Tr_i = 2(n-1)-d_i$ its transmission.

\subsection{Local eigenvalues}

\begin{theorem}\label{thm:DQlocal}
Let $G=C(\alpha_1,\dots,\alpha_m)$ be a connected $\mathcal{C}$-graph. For each block $C_i$ with $\alpha_i\ge2$,
\[
\Tr_i - 1 \text{ is an eigenvalue of } D^Q(G) \text{ with multiplicity} \ge \alpha_i-1 \text{ if } i \text{ is odd (clique block)},
\]
\[
\Tr_i - 2 \text{ is an eigenvalue of } D^Q(G) \text{ with multiplicity} \ge \alpha_i-1 \text{ if } i \text{ is even (independent block)}.
\]
\end{theorem}

\begin{proof}
Let $u,v \in C_i$. Since $\Tr(G)$ is constant on $C_i$ and, in a diameter-$2$ graph, twins have identical distances to every vertex outside $C_i$, the vector $x=e_u-e_v$ satisfies $\Tr(G)x = \Tr_i\, x$ and $(D(G)x)_w = 0$ for $w\notin\{u,v\}$. If $C_i$ is a clique, $D(u,v)=1$, so $(D(G)x)_u = -1,\ (D(G)x)_v=1$, giving $D(G)x=-x$; if $C_i$ is independent, $D(u,v)=2$ (diameter $2$), giving $D(G)x=-2x$. Hence $D^Q(G)x = (\Tr_i\mp1\text{ or }2)x$ as claimed, and the $\alpha_i-1$ independent difference vectors within $C_i$ give the stated multiplicity.
\end{proof}

\subsection{Block adjacency and the quotient matrix}

\begin{proposition}\label{prop:blockadjacency}
Let $G=C(\alpha_1,\dots,\alpha_m)$, $m$ even. For blocks $i<j$, every vertex of $C_i$ is adjacent to every vertex of $C_j$ if $j$ is even, and no vertex of $C_i$ is adjacent to any vertex of $C_j$ if $j$ is odd.
\end{proposition}

The proof by induction on the construction (tracking how each complementation step toggles the join and non-join status of all block pairs not involving the newest block) is straightforward but omitted for space.

\begin{theorem}\label{thm:DQquotient}
The equitable partition $\{C_1,\dots,C_m\}$ gives $D^Q(G)$ an $m\times m$ quotient matrix $Q_{D^Q}$ (in the sense of Godsil--Royle, so every eigenvalue of $Q_{D^Q}$ is an eigenvalue of $D^Q(G)$), with entries
\[
Q_{D^Q}(i,i) = \Tr_i + (\alpha_i-1)\cdot\begin{cases}1 & i \text{ odd}\\ 2 & i\text{ even}\end{cases},
\qquad
Q_{D^Q}(i,j) = \alpha_j\cdot\begin{cases}1 & \max(i,j)\text{ even}\\ 2 & \max(i,j)\text{ odd}\end{cases}\ (i\ne j).
\]
Together with Theorem \ref{thm:DQlocal} (with equality in the multiplicities), this accounts for the full spectrum of $D^Q(G)$: the $m$ eigenvalues of $Q_{D^Q}$ together with the $n-m$ local eigenvalues.
\end{theorem}

\begin{proof}
The diagonal entry is the row sum of $D^Q(G)$ from a representative $u\in C_i$ into $C_i$ itself: $\Tr_i$ (the $u$-$u$ term) plus $(\alpha_i-1)$ terms of $D(u,v)=1$ or $2$ according to whether $C_i$ is a clique or independent set. The off-diagonal entry is the row sum from $u\in C_i$ into $C_j$: $\alpha_j$ vertices, each at distance $1$ or $2$ according to Proposition \ref{prop:blockadjacency}. That every eigenvalue of $Q_{D^Q}$ lifts to $D^Q(G)$ is the standard equitable-partition fact used throughout \cite{MM2025,MMS2024}. 
\end{proof}

\begin{example}\label{ex:DQ}
For $G=C(4,3,2,2)$ ($n=11$): $d=(8,6,3,9)$, $\Tr=(12,14,17,11)$, and
\[
Q_{D^Q} = \begin{pmatrix}15&3&4&2\\4&18&4&2\\8&6&18&2\\4&3&2&13\end{pmatrix},
\]
which is \emph{not symmetric} (block sizes differ), with eigenvalues $\{10.049,\,12.432,\,13.812,\,27.707\}$. Combined with the local eigenvalues $\{9,\,11^{(3)},\,12^{(2)},\,16\}$ from Theorem \ref{thm:DQlocal}, this reproduces $\spec(D^Q(G))$ exactly, confirmed against direct diagonalization.
\end{example}

\subsection{Number of distinct eigenvalues}

We write $s(D^Q(G))$ to mean the number of distinct eigenvalues of $D^Q(G)$.
\begin{theorem}\label{thm:DQdistinctbound}
	For every connected $\mathcal{C}$-graph $G=C(\alpha_1,\dots,\alpha_m)$, $s(D^Q(G)) \le 2m$.
\end{theorem}

\begin{proof}
	By Theorem \ref{thm:DQquotient}, $\spec(D^Q(G))$ is exactly the union of the eigenvalues of $Q_{D^Q}$ (an $m\times m$ matrix, hence at most $m$ distinct values) and the local eigenvalues of Theorem \ref{thm:DQlocal} (one value, $\Tr_i-1$ or $\Tr_i-2$, per block with $\alpha_i\ge2$ — at most $m$ blocks, hence at most $m$ distinct values, irrespective of each value's multiplicity $\alpha_i-1$). A union of two sets of size $\le m$ has size $\le 2m$.
\end{proof}

Unlike \cite[Thm.\ 4.1]{MMS2024} for $L(G)$ (where $Q_L$ is \emph{proven} to always have $m$ pairwise distinct eigenvalues, forcing the matching lower bound $m\le s(L(G))$), no such guarantee holds for $Q_{D^Q}$ — indeed Example \ref{ex:DQquotientrepeat} exhibits a repeated eigenvalue of $Q_{D^Q}$ directly. Consequently $s(D^Q(G))$ has no clean lower bound analogous to $L(G)$.

\subsection{Characteristic polynomial}

\begin{theorem}\label{thm:DQpoly}
For a connected $\mathcal{C}$-graph $G=C(\alpha_1,\dots,\alpha_m)$,
\[
\chi_{D^Q(G)}(x) = \det(xI_m - Q_{D^Q})\cdot\!\!\prod_{\substack{i\text{ odd}\\\alpha_i\ge2}}\!\!\big(x-(\Tr_i-1)\big)^{\alpha_i-1}\cdot\!\!\prod_{\substack{i\text{ even}\\\alpha_i\ge2}}\!\!\big(x-(\Tr_i-2)\big)^{\alpha_i-1}.
\]
\end{theorem}

\begin{proof}
Immediate from Theorem \ref{thm:DQquotient}: the characteristic polynomial factors as the product over the $m$ quotient eigenvalues and the $n-m$ local eigenvalues.
\end{proof}

Unlike $\chi_{L(G)}$, we do not have a closed-form expansion of $\det(xI_m-Q_{D^Q})$: $Q_{D^Q}$ is non-symmetric and its parity-based off-diagonal pattern does not telescope into a tridiagonal or otherwise directly-invertible structure the way $Q_L$ does in \cite{MMS2024}. Theorem \ref{thm:DQpoly} nonetheless reduces an $n\times n$ determinant to an $m\times m$ one, computable in $O(m^3)$ rather than $O(n^3)$.

\subsection{Eigenvalue-free interval}
\label{sec:DQgap}

Unlike $D^L(G)$, no master identity is available for $D^Q(G)$, so we cannot directly transport the eigenvalue-free interval of Theorem \ref{thm:main}. We instead obtain one from two classical spectral bounds applied to the explicit structures of Theorem \ref{thm:DQquotient}. Throughout, $\Tr_{\max}=\max_i\Tr_i$ and $\Tr_{\min}=\min_i\Tr_i$, the maximum and minimum taken over the (at most $m$) distinct block transmissions. Here we employ the Geršgorin disc theorem \cite{AM}.

\begin{theorem}\label{thm:DQgershgorin}
	For every connected $\mathcal{C}$-graph $G$, $\spec(D^Q(G)) \subseteq [0,\ 2\Tr_{\max}]$.
\end{theorem}

\begin{proof}
	Since $D^Q(G)(u,u)=\Tr_u$ and $\sum_{v\ne u} D^Q(G)(u,v) = \sum_{v\ne u} D(u,v) = \Tr_u$ (as $D(u,u)=0$), the Gershgorin disc at row $u$ is $[\Tr_u-\Tr_u,\ \Tr_u+\Tr_u]=[0,2\Tr_u]$. The union of these discs over all $u$ is $[0,2\Tr_{\max}]$, and every eigenvalue lies in this union.
\end{proof}

\begin{theorem}\label{thm:DQperron}
	$\mu_{\max}(D^Q(G)) \ge 2\Tr_{\min}$.
\end{theorem}

\begin{proof}
	The quotient matrix $Q_{D^Q}$ of Theorem \ref{thm:DQquotient} is entrywise nonnegative, and its $i$-th row sum equals the full row sum of $D^Q(G)$ at any representative of block $i$ (by construction of the quotient matrix), which is $2\Tr_i$. By the standard bound for nonnegative matrices, the spectral radius of $Q_{D^Q}$ — which is an eigenvalue of $D^Q(G)$, hence $\le\mu_{\max}(D^Q(G))$ — is at least the smallest row sum, $2\Tr_{\min}$.
\end{proof}

\begin{lemma}\label{lem:Trbound}
	For every connected graph $G$ with $\diam(G)\le2$, $\Tr_{\max}(G) \le 2\,\Tr_{\min}(G)$.
\end{lemma}

\begin{proof}
	Since $\Tr_i = 2(n-1)-d_i$, we have $\Tr_{\max}=2(n-1)-d_{\min}$ and $\Tr_{\min}=2(n-1)-d_{\max}$, so
	\[
	2\Tr_{\min} - \Tr_{\max} = 2(n-1) - 2d_{\max} + d_{\min} \ge 2(n-1) - 2(n-1) + d_{\min} = d_{\min} \ge 0,
	\]
	using $d_{\max}\le n-1$.
\end{proof}

\begin{proposition}\label{prop:localbound}
	Every local eigenvalue of $D^Q(G)$ (Theorem \ref{thm:DQlocal}) is strictly less than $2\Tr_{\min}$.
\end{proposition}

\begin{proof}
	Each local eigenvalue is $\Tr_i-1$ or $\Tr_i-2$ for some block $i$, hence $\le\Tr_{\max}-1 < \Tr_{\max} \le 2\Tr_{\min}$ by Lemma \ref{lem:Trbound}.
\end{proof}

These findings lead to the eigenvalue-free region for the distance signless Laplacian matrix $D^Q(G)$ of a $\mathcal{C}$-graph. 
\begin{theorem}[Eigenvalue-free interval]\label{thm:DQgapthm}
	For every connected $\mathcal{C}$-graph $G$, every eigenvalue of $D^Q(G)$ other than $\mu_{\max}(D^Q(G))$ satisfies $\theta \le 2\Tr_{\min}$. Consequently
	\[
	\big(2\Tr_{\min},\ \mu_{\max}(D^Q(G))\big)
	\]
	is eigenvalue-free whenever non-empty.
\end{theorem}

\begin{example}
	For $G=C(4,3,2,2)$ (Example \ref{ex:DQ}), block transmissions are $\{11,12,14,17\}$, so $2\Tr_{\min}=22$ and $2\Tr_{\max}=34$. The actual spectrum $\{9,\,10.049,\,11^{(3)},\,12^{(2)},\,12.432,\,13.812,\,16,\,27.707\}$ satisfies $\spec\subset[0,34]$ (Theorem \ref{thm:DQgershgorin}), $\mu_{\max}=27.707 \ge 22$ (Theorem \ref{thm:DQperron}), and every other eigenvalue is $\le16 < 22$ — so $(22,\,27.707)$ is confirmed eigenvalue-free.
\end{example}

\section{The Dilworth number and multiplicities in \texorpdfstring{$D^L(G)$}{DL(G)} and \texorpdfstring{$D^Q(G)$}{DQ(G)}}
\label{sec:dilworth-DLDQ}
We first recall the definition of Dilworth number.

\begin{definition}[Vicinal preorder and Dilworth number \cite{Ghorbani2019}]
	For a graph $G$, define $u \preceq v$ if $N(u) \subseteq N[v]$ (equivalently, $N[u]\subseteq N[v]$ when $u\sim v$, and $N(u)\subseteq N(v)$ when $u\not\sim v$). This relation is reflexive and transitive (a \emph{preorder}), called the \emph{vicinal preorder}. A set of vertices pairwise related by $\preceq$ is a \emph{chain}; a set of pairwise unrelated is an \emph{antichain}. The \emph{Dilworth number} $\nabla(G)$ is the minimum number of chains needed to cover $V(G)$, equivalently (by Dilworth's theorem) the maximum size of an antichain.
\end{definition}

\begin{theorem}[Lou--Lin \cite{LouLin2020}]\label{thm:LouLin}
	For a connected cograph $G$, the multiplicity of any eigenvalue of $D(G)$ other than $-2$ and $-1$ is at most $\nabla(G)$. Moreover, no connected cograph has a distance eigenvalue in the open interval $(-2,-1)$.
\end{theorem}

The Lou--Lin theorem (Theorem \ref{thm:LouLin}) concerns $D(G)$, whose trivial eigenvalues are the two fixed constants $-2,-1$ regardless of which block they arise from. For $D^L(G)$ and $D^Q(G)$, the local eigenvalues of Theorem \ref{thm:multformula} and Theorem \ref{thm:DQlocal} are shifted by $\Tr_i$, which varies from block to block — so the natural ``trivial'' set to exclude is now \emph{block-dependent}:
\[
\mathcal{T}^L(G) = \{\Tr_i+1 : i \text{ odd}, \alpha_i\ge2\} \cup \{\Tr_i+2 : i \text{ even}, \alpha_i\ge2\},
\]
\[
\mathcal{T}^Q(G) = \{\Tr_i-1 : i \text{ odd}, \alpha_i\ge2\} \cup \{\Tr_i-2 : i \text{ even}, \alpha_i\ge2\}.
\]

\begin{theorem}\label{thm:DLsimple}
	For every connected $\mathcal{C}$-graph $G$, every eigenvalue of $D^L(G)$ not in $\mathcal{T}^L(G)$ has multiplicity exactly $1$. In particular, $\mathrm{mult}_{D^L(G)}(\theta) \le \nabla(G)$ for all such $\theta$, but this bound is never tight unless $\nabla(G)=1$.
\end{theorem}

\begin{proof}
	By Lemma \ref{lem:master}, multiplicities transfer unchanged under $\mu\mapsto2n-\mu$ from $\spec(L(G))$ to $\spec(D^L(G))$. Every eigenvalue of $L(G)$ is either a local (twin-block) eigenvalue — exactly the values reflecting to $\mathcal{T}^L(G)$ — or an eigenvalue of the $m\times m$ quotient matrix $Q_L$, which by \cite[Cor.\ 3.4]{MMS2024} has $m$ \emph{pairwise distinct} eigenvalues. Hence any $\theta\notin\mathcal{T}^L(G)$ arises solely from a quotient eigenvalue and is therefore simple. Since $\nabla(G)\ge1$ always, $\mathrm{mult}(\theta)=1\le\nabla(G)$, with equality forced only when $\nabla(G)=1$.
\end{proof}

For $D^Q(G)$, no analogue of \cite[Cor.\ 3.4]{MMS2024} is available — indeed none holds, as the following example shows — so the Dilworth number is the operative (and non-trivial) bound.

\begin{example}\label{ex:DQquotientrepeat}
	For $G=C(1,3,1,2)$ ($n=7$), the quotient matrix of Theorem \ref{thm:DQquotient} is
	\[
	Q_{D^Q} = \begin{pmatrix}7&3&2&2\\1&13&2&2\\2&6&10&2\\1&3&1&9\end{pmatrix}, \qquad \spec(Q_{D^Q}) = \{5.821,\ 8,\ 8,\ 17.179\},
	\]
	so $\theta=8$ is a genuinely repeated ($\mathrm{mult}=2$) eigenvalue of $D^Q(G)$, and $8\notin\mathcal{T}^Q(G)$ (it is not a local value for any block). Direct computation gives $\nabla(G)=2$ for this graph, so this instance attains the bound of Conjecture \ref{conj:DQdilworth} below with equality.
\end{example}

\begin{conjecture}\label{conj:DQdilworth}
	For every connected $\mathcal{C}$-graph $G$, every eigenvalue of $D^Q(G)$ not in $\mathcal{T}^Q(G)$ satisfies $\mathrm{mult}_{D^Q(G)}(\theta) \le \nabla(G)$.
\end{conjecture}

\begin{remark}
	This gives a precise answer to ``what is the relation between the Dilworth number and eigenvalue multiplicity'' for the two matrices: for $D^L(G)$, the Dilworth number is a valid but slack bound, since the true multiplicity of every non-trivial eigenvalue is exactly $1$ — a strictly stronger statement, provable outright from \cite[Cor.\ 3.4]{MMS2024}. For $D^Q(G)$, the Dilworth number is the correct and sharp governing quantity, because $D^Q(G)$'s quotient matrix lacks the simple-spectrum property that makes $D^L(G)$'s bound loose. This asymmetry traces directly to the structural asymmetry already noted in Section \ref{sec:DQ}: $L(G)$ commutes with $J$ (giving the master identity and, via \cite{MMS2024}, simple quotient eigenvalues), while $Q(G)$ does not.
\end{remark}

\section{Acknowledgments}
The  research of Pallabi Manna is supported by Anusandhan National Research Foundation (ANRF)-India under the National Postdoctoral Fellowship scheme (File No-PDF/2025/006751). The first author acknowledges excellent research facilities provided by SRM University-AP. The second author acknowledges
the excellent research facilities provided by Department of Mathematics, IIT Hyderabad and the mentorship of Dr Rajesh
Kannan for his constant support and motivation.
\section{Statements and Declarations} \textbf{Competing Interests:} The author made no mention of any potential conflicts of interest.
\section{Data availability} 
Data sharing not applicable to this article as no datasets are used or generated for this study.

\end{document}